\documentclass[a4paper,11pt,oneside]{article}
\usepackage{amsmath,amssymb,amsthm}

\newcommand{\Z}{\mathbb{Z}}

\newcommand{\pres}[2]{\langle {#1}\ |\ {#2} \rangle}
\newtheorem{theorem}{Theorem}
\newtheorem{conjecture}[theorem]{Conjecture}
\newtheorem{lemma}[theorem]{Lemma}
\newtheorem{corollary}[theorem]{Corollary}
\numberwithin{theorem}{section}
\theoremstyle{definition}
\newtheorem{defn}[theorem]{Definition}
\newtheorem{example}[theorem]{Example}

\begin{document}
\title{Redundant relators in cyclic presentations of groups}
\author{Ihechukwu Chinyere and Gerald Williams\thanks{This work was supported by the Leverhulme Trust Research Project Grant RPG-2017-334 and partially supported by the grant 346300 for IMPAN from the Simons Foundation and the matching 2015-2019 Polish MNiSW fund.}}

\maketitle

\begin{abstract}
  A cyclic presentation of a group is a presentation with an equal number of generators and relators that admits a particular cyclic symmetry. We characterise the orientable, non-orientable, and redundant cyclic presentations and obtain concise refinements of these presentations. We show that the Tits alternative holds for the class of groups defined by redundant cyclic presentations and that
  if the number of generators of the cyclic presentation is greater than two then the corresponding group is large. Generalizing and extending earlier results of the authors we describe the star graphs of orientable and non-orientable cyclic presentations and classify the cyclic presentations whose star graph components are pairwise isomorphic incidence graphs of generalized polygons, thus classifying the so-called $(m,k,\nu)$-special cyclic presentations.
\end{abstract}


\noindent \textbf{MSCs:} 20F05, 20E05, 05E18 (primary); 20E42, 20F67, 51E24, 52B05, 57M07 (secondary).

\section{Introduction}\label{sec:intro}

A cyclic presentation is a group presentation with an equal number of generators and relators that admits a particular cyclic symmetry and the corresponding group is a cyclically presented group \cite{Johnson97} (see Section \ref{sec:prelims} for definitions and background). Cyclic presentations in which certain relators are cyclic permutations of other relators or their inverses are called redundant, and in the latter situation they are non-orientable. A presentation with no such relators is called concise and a concise presentation obtained from a presentation by removing redundant relators is called a concise refinement.

In this paper we classify the redundant cyclic presentations, we show that the class of groups defined by such presentations satisfies the Tits alternative, and we classify the (redundant and concise) $(m,k,\nu)$-special cyclic presentations (that is, the cyclic presentations with length $k$ relators, where the star graph has $\nu$ components, each of which is the incidence graph of a generalized $m$-gon).

The paper is organised as follows. In Section \ref{sec:classificationredundant} we observe that redundant cyclic presentations are ubiquitous, in the sense that for any $n\geq 2$ and any non-empty word $u$ in the free group of rank $n$, it is possible to construct both orientable and non-orientable redundant cyclic presentations using these $n,u$. We classify the orientable (redundant and concise) presentations and classify the non-orientable cyclic presentations; in each case we obtain concise refinements of these presentations.
In Section \ref{sec:deficiencylargeness} we show that if a redundant cyclic presentation has more than two generators then the group it defines is large and, in particular, motivated by \cite[Problem 2]{EdjvetVdovina} (which asks which groups defined by $(m,k$)-special cyclic presentations are large), we observe that groups defined by redundant $(m,k,\nu)$-special cyclic presentations are large. We show that the Tits alternative holds for the class of groups defined by redundant cyclic presentations and we investigate which redundant cyclic presentations with two generators define groups that contain a non-abelian free subgroup, providing a classification in the orientable case.
In Section \ref{sec:stargraphs} we generalize \cite[Theorem 3.3]{ChinyereWilliamsGP} to describe the star graphs of orientable and non-orientable cyclic presentations.
In Section \ref{sec:specialpresentationsclassification} we extend the classification in \cite[Sections 5,6]{ChinyereWilliamsGP} to classify the (redundant and concise) $(m,k,\nu)$-special cyclic presentations. In particular, we show that if a cyclic presentation is $(m,k,\nu)$-special, where $m\geq 3$ then $m=3$, the presentation is orientable, and the defining word is positive or negative, and we show that if a redundant cyclic presentation is $(m,k,\nu)$-special then $1/m+2/k<1$, and hence defines a non-elementary hyperbolic group.

\section{Preliminaries}\label{sec:prelims}

\subsection{Presentations and cyclic presentations of groups}\label{sec:wordsinpresentations}

Given a positive integer $n$, let $F_n$ be the free group with basis $X=\lbrace x_0, \ldots, x_{n-1}\rbrace$. A non-empty word $w\in F_n$ is said to be \em positive \em (resp.\,\em negative\em) if all of the exponents of generators are positive (resp.\,negative). (In particular, generators $x_i$ are positive, and their inverses $x_i^{-1}$ are negative.) We shall say that a word $w$ of length at least 2 is \em alternating \em if it has no subword of the form $(x_ix_j)^{\pm 1}$ and that it is \em cyclically alternating \em if it has no cyclic subword of that form. Thus a word is cyclically alternating if and only if it is alternating and has even length. A word $w$ is \em reduced \em if it does not contain a subword of the form $x_ix_i^{-1}$ or $x_i^{-1}x_i$; it is \em cyclically reduced \em if all cyclic permutations of it are reduced. If $w\in F_n$ is a cyclically reduced, non-empty, word then the unique word $v\in F_n$ such that $w=v^p$ with $p$ maximal is called the \em root \em of $w$. We shall write $l(w)$ to denote the length of $w$ in $F_n$. Throughout this article equality of words will refer to equality of elements of the free group $F_n$.

When considering group presentations $\pres{X}{R}$, $R$ will be a set of relators (so does not contain any relator more than once), all of whose elements are cyclically reduced. Let $\theta: F_n\rightarrow F_n$ be the \em shift automorphism \em given by $\theta(x_i)=x_{i+1}$ where (as throughout this article) subscripts are taken modulo $n$. Given a non-empty cyclically reduced word $w$ representing an element in $F_n$, the words $\theta^i(w)$, $0\leq i<n$ (resp.\,$1\leq i<n$),  are the \em shifts \em (resp.\,\em proper shifts\em) of $w$, the presentation
\[P_n(w)=\pres{x_0,\ldots, x_{n-1}}{w,\theta(w), \ldots, \theta^{n-1}(w)}\]
is a \em cyclic presentation, \em the group $G_n(w)$ it defines is a \em cyclically presented group \em and $w$ is the \em defining word\em. The shift automorphism $\theta$ satisfies $\theta^n=1$ and the resulting $\Z_n$-action on $G_n(w)$ determines the \em shift extension \em $E=G_n(w)\rtimes_\theta \Z_n$, which admits a presentation of the form $\pres{x,t}{t^n,W(x,t)}$ where $W(x,t)$ is obtained by rewriting $w$ in terms of the substitutions $x_i=t^ixt^{-i}$, $0\leq i<n$ (see, for example, \cite[Theorem 4]{JWW}).

For $1\leq t\leq n$ we define the \em $t$-truncation \em  of $P_n(w)$ to be the presentation
\[P_{n,t}(w)=\pres{x_0,\ldots, x_{n-1}}{w,\theta(w), \ldots, \theta^{t-1}(w)}\]
and denote by $G_{n,t}(w)$ the group that it defines. Let $\phi: F_n\rightarrow F_n$ be the cyclic permutation function that cyclically permutes a word by one generator (or the inverse of a generator), with inverse $\phi^{-1}$. Then $\phi$ and $\theta$ commute and, if $w$ is a cyclically reduced word of length $k$, then $\phi^{k}(w)=w$. It follows that if $w$ is a cyclically reduced word of length $k$ that is equal to a cyclic permutation of the shift $\theta^h(w)$ then a cyclic permutation of $w$ is equal to the shift $\theta^{(n,h)}(w)$ and if a shift of $w$ is equal to a cyclic permutation $\phi^t(w)$ of $w$ then a shift of $w$ is equal to the cyclic permutation $\phi^{(k,t)}(w)$. (To see the first claim, let $\alpha,\beta\in \Z$ satisfy $\alpha h + \beta n=(n,h)$ and observe that $w$ is equal to a cyclic permutation of $\theta^{\alpha h}(w)=\theta^{\alpha h+\beta n}(w)=\theta^{(n,h)}(w)$. Similar arguments hold for second claim.) Moreover, $\phi^r(v^p)=(\phi^r(v))^p$ and $\theta^h(v^p)=(\theta^h(v))^p$ for any $h,r,p\geq 1$ so if $w=v^p$ where $v$ is the root of $w$ then $\theta^h(w)=\phi^r(w)$ if and only if $\theta^h(v)=\phi^r(v)$. We will write $\iota(w),\tau(w)$ to denote the initial and terminal letters of a word $w$, respectively.

Following \cite[page 160]{BogleyShift}, given a group presentation $P=\pres{X}{R}$, an element $r\in R$ is said to be \em freely redundant \em if it is freely trivial or if there exists another element $s\in R$ such that $r$ and $s$ are elements of the free group with basis $X$ and either $r$ is freely conjugate to $s$ or $r$ is freely conjugate to $s^{-1}$ (that is, either $r$ is a cyclic permutation of $s$ or of $s^{-1}$). A presentation is said to be \em redundant \em if it contains a freely redundant relator and is \em concise \em otherwise (\cite[page 4]{CCH}). (Concise presentations are referred to as \em slender presentations \em in \cite[page 5]{BogleyPride} or \em irredundant \em presentations in \cite[page 82]{Rutter}). If a presentation $P'$ is obtained from a presentation $P$ by removing freely redundant relators then we say that $P'$ is a \em refinement \em of $P$ and if, in addition, $P'$ is concise we say that it is a \em concise refinement \em of $P$  (compare \cite[page 4]{CCH}). A cyclic presentation $P_n(w)$ is \em orientable \em if $w$ is not a cyclic permutation of the inverse of any of its shifts (\cite[page~155]{BogleyShift}) and is \em non-orientable \em otherwise. Thus, an orientable cyclic presentation is redundant if and only if $w$ is equal to a cyclic permutation of one of its proper shifts, and if a cyclic presentation is non-orientable then it is redundant. As examples, the cyclic presentation $P_2(x_0x_1)$ is orientable and redundant, $P_3(x_0x_1)$ is orientable and concise, and $P_2(x_0x_1^{-1})$ is non-orientable (and therefore redundant).

The \em deficiency \em of the presentation $P=\pres{X}{R}$ is defined as $\mathrm{def}(P)=|X|-|R|$ and the \em deficiency \em of a group $G$, $\mathrm{def}(G)$, is defined to be the maximum of the deficiencies of all finite presentations defining $G$. A group $G$ is \em large \em if it has a finite index subgroup that has a non-abelian free homomorphic image and it is \em SQ-universal \em if every countable group embeds in a quotient of $G$; every large group is SQ-universal, and hence contains a non-abelian free subgroup \cite{Pride80}. A class of groups is said to satisfy the \em Tits alternative \em if every group in that class either contains a non-abelian free subgroup or is virtually solvable.

\subsection{Star graphs}\label{sec:stargraph}

Let $P = \pres{X}{R}$ be a group presentation and let $\tilde{R}$ denote the symmetrized closure of $R$; that is, the set of all cyclic permutations of elements in $R\cup R^{-1}$. The \em star graph \em of $P$ is the undirected vertex-labelled graph $\Gamma$ where the vertex set is in one-one correspondence with $X\cup X^{-1}$, vertices are labelled by the corresponding element of $X\cup X^{-1}$ and where there is an edge  joining vertices labelled $x$ and $y$ for each distinct word $xy^{-1}u$ in $\tilde{R}$~\cite[page~61]{LyndonSchupp}. Such words occur in pairs, that is $xy^{-1}u\in \tilde{R}$ implies that $yx^{-1}u^{-1}\in \tilde{R}$. These pairs are called \em inverse pairs \em and the two edges corresponding to them are identified in $\Gamma$. It follows that replacing any relator of a presentation by its root, or removing a redundant relator from a presentation, leaves the star graph unchanged; in particular, the star graphs of a presentation and any concise refinement of it are equal. We
refer to vertices in $X$ as \em positive \em vertices and vertices in $X^{-1}$ as \em negative \em vertices.

We now set out our graph theoretic terminology. We allow graphs to have loops and to have more than one edge joining a pair of vertices. Given a graph $\Gamma$ we write $V(\Gamma)$ to denote its vertex set. If $\Gamma$ is bipartite with vertex partition $V(\Gamma)=V_1\cup V_2$ where each edge connects a vertex in $V_1$ to a vertex in $V_2$ then $V_1,V_2$ are called the \em parts \em of $V(\Gamma)$. Two adjacent vertices are said to be \em neighbours \em and the set of neighbours of a vertex $v$ in a graph $\Gamma$ is denoted $N_\Gamma(v)$. A graph $\Gamma$ is \em $r$-regular
\em if $|N_\Gamma(v)|=r$ for all $v\in V(\Gamma)$ and it is \em regular \em if it is $r$-regular for some $r$.

A \em path \em of \em length \em $l$ in $\Gamma$ is a sequence of vertices $(u=u_0,u_1,\ldots ,u_l=v)$ with edges $u_i-u_{i+1}$ for each $0\leq i<l$; it is a \em closed path \em if $u=v$. The path is \em reduced \em if the edge $u_{i+1}-u_{i+2}$ is not equal to the edge $u_{i+1}-u_{i}$ ($0\leq i<l-1$). The \em distance \em  $d_\Gamma (u,v)$ between vertices $u,v$ of $\Gamma$ is $l\geq 0$ if there is a path of length $l$ from $u$ to $v$, but no shorter path, and $d_\Gamma(u,v)=\infty$ if there is no path from $u$ to $v$. The \em girth\em, $\mathrm{girth}(\Gamma)$ of a graph $\Gamma$ is the length of a reduced closed path of minimal length, if $\Gamma$ contains a reduced closed path, and $\mathrm{girth}(\Gamma)=\infty$ otherwise. The \em diameter\em, $\mathrm{diam}(\Gamma)$ of a graph $\Gamma$ is the greatest distance  between any pair of vertices of the graph (which may be infinite). If $\Gamma$ is a graph with finite girth then $\mathrm{girth}(\Gamma)\leq 2\mathrm{diam}(\Gamma)+1$. For an integer $n\geq 2$ and a (multi-)set $A \subseteq \{0, 1, \ldots , n-1\}$, the \em circulant graph \em $\mathrm{circ}_n(A)$ is the graph with vertices $v_0, \ldots , v_{n-1}$ and edges $v_i -v_{i+a}$ for all $0 \leq i < n$, $a \in A$ (subscripts $\bmod n$). We define the graph $\mathrm{circ}'_n(A)$ to be $\mathrm{circ}_n(A)$ if $n/2 \not \in A$ and to be $\mathrm{circ}_n(A)$ with exactly one edge $v_i-v_{i+n/2}$ removed for each $n/2\leq i <n$ otherwise. For later use we note that $\mathrm{circ}_n(A)$ is the complete bipartite graph $K_{n/2,n/2}$ if and only if $n/2$ is even and $A=\{ \pm 1, \pm 3 , \ldots , \pm(n/2-1)\}$ and $\mathrm{circ}'_n(A)$ is the complete bipartite graph $K_{n/2,n/2}$ if and only if $n/2$ is odd and $A=\{ \pm 1, \pm 3 , \ldots , \pm (n/2-2), n/2 \}$.

\subsection{Special presentations}\label{sec:specialpresentation}

An $(m,k,\nu)$-special presentation is a group presentation in which the relators have length $k$ and whose star graph has $\nu$ isomorphic components, each of which is the incidence graph of a generalized $m$-gon. Formally:

\begin{defn}[{\cite[Definition~2.1]{ChinyereWilliamsGP}}]\label{def:mkalphaspecialpres}
Let $m\geq 2, k\geq 3,\nu\geq 1$. A  finite group presentation $P = \pres{X}{R}$ is said to be \em $(m, k,\nu)$-special \em if the following conditions hold:
\begin{itemize}
  \item[(a)] the star graph $\Gamma$ of $P$ has $\nu$ isomorphic components, each of which is a connected, bipartite graph of diameter $m$ and girth $2m$ in which each vertex has degree at least $3$;
  \item[(b)] each relator $r\in R$ has length $k$;
  \item[(c)] if $m=2$ then $k\geq 4$.
\end{itemize}
\end{defn}

This generalizes the concept of $(m,k)$-special cyclic presentations, introduced in \cite{EdjvetVdovina}, which corresponds to the case $\nu=1$, which in turn generalize the concept of special presentations, introduced in~\cite{Howie89} (which corresponds to the case $m=k=3$). The concise cyclic presentations that are $(m,k,\nu)$-special were classified in \cite{ChinyereWilliamsGP}. We refer the reader to \cite{EdjvetVdovina,ChinyereWilliamsGP} for background and further references on properties of $(m,k,\nu)$-special presentations and the groups they define.

Note that if a presentation is $(m,k,\nu)$-special then it has at least 3 generators. As in \cite[Proof of Theorem 2]{EdjvetVdovina}, a group $G$ defined by an $(m,k,\nu)$-special presentation with $2/k+1/m<1$ is non-elementary hyperbolic, and hence SQ-universal. If $w=v^p$ where $v$ is the root of $w$ then the star graph $\Gamma$ of $P_n(w)$ is equal to the star graph of $P_n(v)$. Moreover, if $v$ has length 2 then the vertices of $\Gamma$ have degree at most 2 so neither $P_n(v)$ nor $P_n(w)$ are $(m,k,\nu)$-special. Therefore $P_n(w)$ is $(m,pk,\nu)$-special if and only if $P_n(v)$ is $(m,k,\nu)$-special. Thus, in classifying $(m,k,\nu)$-special cyclic presentations $P_n(w)$ we can assume that $w$ is not a proper power. For our characterisation of $(3,k,\nu)$-special cyclic presentations we recall that a set of $k$ integers $d_1,\ldots ,d_k$ is called a \em perfect difference set \em (of order $k$) if among the $k(k-1)$ differences $d_i-d_j$ ($i\neq j$) each of the residues $1,2,\ldots , (k^2-k) \bmod (k^2-k+1)$ occurs exactly once.

\section{Classification of redundant cyclic presentations}\label{sec:classificationredundant}

\subsection{Classification of orientable redundant cyclic presentations}\label{sec:classificationorientableredundantNEW}

Suppose that $P_n(w)$ is orientable and that $v$ is the root of $w$ and recall from Section \ref{sec:wordsinpresentations} that a cyclic permutation of $w$ is equal to one of its proper shifts if and only if a cyclic permutation of $v$ is equal to of one of its proper shifts. It follows that $P_{n,t}(w)$ is a concise refinement of $P_n(w)$ if and only if $P_{n,t}(v)$ is a concise refinement of $P_n(v)$. Therefore, in this context we may assume that $w$ is not a proper power.

Observe that any non-empty word $w \in F_n$ has an expression of the form
\begin{alignat}{1}
w=\prod_{i=0}^{n/(n,h )-1} \theta^{ih}(u)\label{eq:thetaimform}
\end{alignat}
for some $u \in F_n$ of length $l(u)\geq 1$ and some $0\leq h<n$ (where $u=w$ in the case $h=0$). Therefore $w$ satisfies $\phi^{l(u)}(w)=\theta^h(w)$, and hence a cyclic permutation of $w$ is equal to the shift $\theta^{(n,h )}(w)$, which is proper if $0< h <n$.

Note that $w$ is positive if and only if $u$ is positive, $w$ is cyclically alternating if and only if $u$ is cyclically alternating, and if $w\neq u$ then $w$ is alternating if and only if it is cyclically alternating. For later use we record that, given a word $w$ of the form (\ref{eq:thetaimform}), every cyclic permutation of $w$ has an expression of the same form, where the length of $u$ and the value of $h $ are preserved.

\begin{lemma}\label{lem:allcycpermshavethisform}
Let $w=\prod_{i=0}^{n/(n,h )-1}\theta^{ih }(u)$ for some reduced word $u$ and for some $0\leq h <n$. Then every cyclic permutation  of $w$ is of the form
\(\prod_{i=0}^{n/(n,h )-1}\theta^{ih }(v)\)
for some reduced word $v$ where $l(v)=l(u)$.
\end{lemma}

\begin{proof}
Let $l=l(u)$ and $0\leq s<l(w)$ and write $s=tl+r$ where $0\leq t<n/(n,h )$, $0\leq r<l$. Then $\phi^s(w)=\prod_{i=0}^{n/(n,h )-1}\theta^{ih }(v)$ where $v=\theta^{th }(u_2)\theta^{(t+1)h }(u_1)$ where $u_1$ is the initial subword of $u$ of length $r$ and $u_2$ is the terminal subword of $u$ of length $l-r$.
\end{proof}
Theorem \ref{thm:classificationoforientableredundantpresentations28092021} shows that, conversely to the observations made above, if $w$ is not a proper power and if a shift of $w$ is equal to a cyclic permutation of $w$, then $w$ is of the form (\ref{eq:thetaimform}) for some $0\leq h <n$. Corollary \ref{cor:conciserefinementorientable} shows that if $u$ is chosen to be the shortest possible then (for the corresponding value of $h $) the truncation $P_{n,(n,h )}(w)$ is a concise refinement of $P_n(w)$.

\begin{theorem}\label{thm:classificationoforientableredundantpresentations28092021}
Let $w\in F_n$ be a non-empty cyclically reduced word of length $k$ that is not a proper power and suppose that some shift of $w$ is equal to $\phi^t(w)$ for some $0\leq t<k$. Then \[w=\prod_{i=0}^{n/(n,h ) -1}\theta^{ih }(u)\]
for some $0\leq h<n$, where $u$ is the initial subword of $w$ of length $(k,t)$.
\end{theorem}

\begin{proof}
Since some shift of $w$ is equal to $\phi^t(w)$, we have $\theta^{h}(w)=\phi^{(k,t)}(w)$ for some $0\leq h<n$. Let $\lambda=k/(k,t)$ and write $w=\prod_{i=0}^{\lambda-1} u_i$ for some words $u_0,\ldots ,u_{\lambda-1}$, each of length $(k,t)$. Then
\[ \prod_{i=0}^{\lambda-1} u_{i+1} = \phi^{(k,t)}(w)= \theta^h(w) = \prod_{i=0}^{\lambda-1} \theta^h(u_i)\]
(subscripts $\bmod\;\lambda$). Thus for each $0\leq i<\lambda$ we have $u_{i+1}=\theta^h(u_i)$ so $u_{i+1}=\theta^{(i+1)h}(u_0)$. In particular, $u_0=u_\lambda=\theta^{\lambda h}(u_0)$, so $\lambda h\equiv 0 \bmod n$ so $\lambda \equiv 0 \bmod n/(n,h)$, so $\lambda = pn/(n,h)$ for some $p\geq 1$. If $p>1$ then $n/(n,h)<\lambda$ and  $u_{n/(n,h)}=\theta^{(n/(n,h))\cdot h}(u_0)=u_0$, so $w$ is a proper power, a contradiction. Therefore $\lambda=n/(n,h)$ and hence %
\[ w=\prod_{i=0}^{\lambda-1} u_i=\prod_{i=0}^{n/(n,h)-1} \theta^{ih}(u)\]
where $u=u_0$, as required.
\end{proof}

\begin{corollary}\label{cor:conciserefinementorientable}
Let $w\in F_n$ be a non-empty cyclically reduced word of length $k$ that is not a proper power and let $u$ be the shortest subword of $w$ such that
\( w=\prod_{i=0}^{n/(n,h) -1}\theta^{ih}(u) \)
for any $0\leq h<n$. Then $P_{n,(n,h)}(w)$ is a concise refinement of $P_n(w)$.
\end{corollary}

\begin{proof}
Observe first that $\theta^h(w)=\phi^{l(u)}(w)$ so $\theta^{(n,h)}(w)$ is a cyclic permutation of $w$ so $P_{n,(n,h)}(w)$ is a refinement of $P_n(w)$. Suppose for contradiction that $P_{n,(n,h )}(w)$ is not concise. Then $\theta^s(w)=\phi^r(w)$ for some $0<s<(n,h)$ , $0\leq r<k$. Let $t=l(u)$ (noting that $t|k$), and let $a,b\in \Z$ satisfy $ar+bt =(r,t)$. Then
\[ \theta^{as+bh}(w)=\theta^{as}(\theta^{bh}(w))=\theta^{as}(\phi^{bt} (w))= \phi^{bt}(\theta^{as}(w))=\phi^{bt}(\phi^{ar}(w))=\phi^{ar+bt}(w)=\phi^{(r,t)}(w).\]
Theorem \ref{thm:classificationoforientableredundantpresentations28092021} then implies $(r,t)\geq t$ so $(r,t)=t$, and hence $r=\lambda t$ for some $\lambda \geq 1$. Therefore $\theta^s(w)=\phi^r(w)=\phi^{\lambda t}(w)=\theta^{\lambda h}(w)$ and hence $s\equiv \lambda h \bmod n$, but $0<s<(n,h)$, a contradiction.
\end{proof}

\subsection{Classification of non-orientable cyclic presentations}\label{sec:classificationnonorientableredundant}

Given any non-empty word $u \in F_n$, where $n\geq 2$ is even, and the word $w=u \theta^{n/2}(u)^{-1}\in F_n$, the cyclic presentation $P_n(w)$ is non-orientable. Lemma~3.6 of \cite{BogleyShift} characterises the non-orientable cyclic presentations $P_n(w)$ as those for which $n$ is even and $w$ is equal to $u\theta^{n/2}(u)^{-1}$ for some reduced word $u$. Unfortunately that statement is not quite correct: a presentation that demonstrates this is  $P_4(w)$ where $w=x_0x_2^{-1}x_3x_1^{-1}$. (However, the remaining results from \cite{BogleyShift} appear to be unaffected.) In Theorem~\ref{thm:classificationofnonorientablepresentations} we instead characterise such presentations as those for which $n$ is even and \em some cyclic permutation of \em $w$ is equal to $u\theta^{n/2}(u)^{-1}$ for some reduced word $u$ and we show that the truncation $P_{n,n/2}(w)$ is a concise refinement of $P_n(w)$.

\begin{theorem}\label{thm:classificationofnonorientablepresentations}
Let $w\in F_n$ be a non-empty cyclically reduced word. The cyclic presentation $P_n(w)$ is non-orientable if and only if $n$ is even and $w$ has a cyclic permutation that is equal to $u\theta^{n/2}(u)^{-1}$ for some reduced word $u\in F_n$, in which case $P_{n,n/2}(w)$ is a concise refinement of $P_n(w)$.
\end{theorem}

\begin{proof}
Suppose that $n$ is even and that $w$ has a cyclic permutation that is equal to $u\theta^{n/2}(u)^{-1}$ for some reduced word $u$. Then $w$ is equal to a cyclic permutation of $\theta^{n/2}(w^{-1})$ so $P_n(w)$ is non-orientable.

Conversely, suppose $\theta^h(w)=\phi^t(w^{-1})$ for some $0\leq h<n$ and some $0\leq t<l(w)$. Note that $h\neq 0$ since the only word that is equal to a cyclic permutation of its inverse is the empty word. Let $w=x_{d_0}^{\epsilon_0}x_{d_1}^{\epsilon_1}\ldots x_{d_{k-1}}^{\epsilon_{k-1}}$ where $\epsilon_j\in \{\pm 1\}$ and $0\leq d_j< n$ for each $0\leq j<k$, $k\geq 1$. Then
\[ x_{h+d_0}^{\epsilon_0}x_{h+d_1}^{\epsilon_1}\ldots x_{h+d_{t-1}}^{\epsilon_{t-1}}\cdot x_{h +d_{t}}^{\epsilon_{t}} \ldots x_{h +d_{k-1}}^{\epsilon_{k-1}}=x_{d_{t-1}}^{-\epsilon_{t-1}}x_{d_{t-2}}^{-\epsilon_{t-2}}\ldots x_{d_0}^{-\epsilon_0}\cdot x_{d_{k-1}}^{-\epsilon_{k-1}}x_{d_{k-2}}^{-\epsilon_{k-2}}\ldots x_{d_{t}}^{-\epsilon_{t}}.\]
Hence by comparing exponents and subscripts we have
\begin{alignat}{1}
d_j&\equiv h+ d_{t-1-j}\bmod n, \quad  \epsilon_j = -\epsilon _{t-1-j} \quad\mathrm{for~all}~0\leq j<t,\label{eq:di-e1}\\
d_j&\equiv h+ d_{k+t-1-j}\bmod n, \quad  \epsilon_j = -\epsilon _{k+t-1-j} \quad\mathrm{for~all}~t\leq j<k.\label{eq:di-e2}
\end{alignat}
In particular since  $d_0\equiv (h+ d_{t-1})$ mod $n$ and $d_{t-1} \equiv ( h+ d_{0})$ mod $n$, we have $2h\equiv 0$ mod $n$, and so $n=2h$.
If $t$ is odd then~(\ref{eq:di-e1}) implies that $\epsilon_{(t-1)/2}=0$, a contradiction, thus $t$ is even, $t=2\mu\geq 0$, say. Similarly $k-t$ is even, $k-t=2\nu\geq 0$, say. Therefore, eliminating $t=2\mu$ and $k=2\mu+2\nu$ equations~(\ref{eq:di-e1}),(\ref{eq:di-e2}) become
\begin{alignat*}{1}
d_j &\equiv h+ d_{2\mu-1-j}\bmod n, \quad  \epsilon_j = -\epsilon _{2\mu-1-j} \quad\mathrm{for~all}~0\leq j<2\mu, \\
d_j &\equiv h+ d_{4\mu+2\nu-1-j}\bmod n, \quad  \epsilon_j = -\epsilon _{4\mu+2\nu-1-j} \quad\mathrm{for~all}~~2\mu\leq j<2\mu+2\nu,
\end{alignat*}
respectively. Therefore
\begin{alignat*}{1}
\allowdisplaybreaks
w
&=
\prod_{j=0}^{\mu-1} x_{d_j}^{\epsilon_j} \cdot
\prod_{j=\mu}^{2\mu-1} x_{d_j}^{\epsilon_j} \cdot
\prod_{j=2\mu}^{2\mu+\nu-1} x_{d_j}^{\epsilon_j} \cdot
\prod_{j=2\mu+\nu}^{2\mu+2\nu-1} x_{d_j}^{\epsilon_j}\\\displaybreak
&=
\prod_{j=0}^{\mu-1} x_{d_j}^{\epsilon_j} \cdot
\prod_{j=0}^{\mu-1} x_{h+d_{\mu-j-1}}^{-\epsilon_{\mu-j-1}} \cdot
\prod_{j=2\mu}^{2\mu+\nu-1} x_{d_j}^{\epsilon_j} \cdot
\prod_{j=2\mu}^{2\mu+\nu-1} x_{h+d_{4\mu+\nu-1-j}}^{-\epsilon_{4\mu+\nu-1-j}}\\
&=
\prod_{j=0}^{\mu-1} x_{d_j}^{\epsilon_j} \cdot
\left( \prod_{j=0}^{\mu-1} x_{h+d_j}^{\epsilon_j} \right)^{-1} \cdot
\prod_{j=2\mu}^{2\mu+\nu-1} x_{d_j}^{\epsilon_j} \cdot
\left( \prod_{j=2\mu}^{2\mu+\nu-1} x_{h+d_j}^{\epsilon_j} \right)^{-1} \\
&= u_1\theta^h(u_1)^{-1}\cdot u_2\theta^h(u_2)^{-1}
\end{alignat*}
where $u_1=\prod_{j=0}^{\mu-1} x_{d_j}^{\epsilon_j}$, $u_2=\prod_{j=2\mu}^{2\mu+\nu-1} x_{d_j}^{\epsilon_j}$. Setting $u=\theta^h(u_1^{-1})u_2$, we have $u\theta^h(u)^{-1}$ is a cyclic permutation of $w$, as required.

Thus $P_{n,n/2}(w)$ is a refinement of $P_n(w)$. Suppose for contradiction that $P_{n,n/2}(w)$ is not concise; then $u\theta^{n/2}(u^{-1})$ is a cyclic permutation of $\theta^i(u\theta^{n/2}(u^{-1}))$ or of $\theta^i( \theta^{n/2}(u)u^{-1} )$ for some $1\leq i <n/2$, but in the latter case $2i\equiv 0 \bmod n$ (as above), a contradiction. Therefore Theorem \ref{thm:classificationoforientableredundantpresentations28092021} implies
\[ u\theta^{n/2}(u^{-1}) = \prod_{i=0}^{n/(n,h)-1} \theta^{ih} (v)\]
for some reduced word $v$ and some $1\leq h<n$. Therefore $v$ is an initial subword of $u$ and $\theta^{n-h}(v)$ is a terminal subword of $\theta^{n/2}(u^{-1})$. Hence $\theta^{n-h }(v)=\theta^{n/2}(v^{-1})$ so $v^{-1}=\theta^{n/2-h}(v)$. Let $x_\iota^{\epsilon_{\iota}}$, $x_\tau^{\epsilon_{\tau}}$ ($0\leq \iota,\tau<n$, $\epsilon_\iota, \epsilon_\tau \in \{\pm 1\}$) be the initial and terminal letters of $v$, respectively. Then by comparing initial and terminal letters of $v^{-1}$ and $\theta^{n/2-h}(v)$ we have $x_\tau^{-\epsilon_\tau}=x_{\iota+n/2-h }^{\epsilon_\iota}$ and $x_{\iota}^{-\epsilon_\iota}=x_{\tau+n/2-h}^{\epsilon_\tau}$. In particular $\tau\equiv \iota+n/2-h$ and $\iota\equiv \tau+n/2-h\bmod n$, so $h=n/2$. Therefore $u\theta^{n/2}(u^{-1})=v\theta^{n/2}(v)$, so $u=v$ and $\theta^{n/2}(u^{-1})=\theta^{n/2}(v)$, a contradiction.
\end{proof}

Thus in considering non-orientable cyclic presentations $P_n(w)$, when relators can be replaced their cyclic permutations, it suffices to consider cyclic presentations $P_n(u\theta^{n/2}(u)^{-1})$ for some reduced word $u$ (even though $w$ itself may not be of the form $u\theta^{n/2}(u)^{-1})$. Note that for such a $w$ and $u$, the word $w$ is non-positive, non-negative, and $w$ is cyclically alternating if and only if $w$ is alternating if and only if $u$ is alternating.

\section{The Tits alternative}\label{sec:deficiencylargeness}

The Tits alternative has been considered for various classes of cyclically presented groups \cite{IsherwoodWilliams, ChinyereWilliamsT5, MohamedWilliams2, BW2, EdjvetWilliams, WilliamsLargeness}. In this section we investigate the Tits alternative for the class of groups defined by redundant cyclic presentations. Observe that solvable groups arise within this class:
the group $\Z$ has orientable and non-orientable redundant cyclic presentations $P_2(x_0x_1), P_2(x_0x_1^{-1})$, the group $\Z^2$ has the non-orientable cyclic presentation $P_2(x_0x_1x_0^{-1}x_1^{-1})$, the Baumslag-Solitar group $BS(1,-1)$ (the fundamental group of the Klein bottle) has the orientable redundant cyclic presentation $P_2(x_0^2)$. Groups defined by redundant cyclic presentations have positive deficiency; by \cite{BaumslagPride} a virtually solvable group of positive deficiency has deficiency one; and a group of deficiency one is solvable if and only if it is isomorphic to a group of the form $G_k=\pres{a,b}{bab^{-1}=a^k}$ for some $k\in \Z$ \cite[Theorem 1]{Wilson}. Note that $G_0\cong \Z$, $G_1\cong \Z^2$ and $G_{-1}\cong BS(1,-1)$; we expect that if $k\not\in \{-1,0,1\}$ then $G_k$ does not have a cyclic presentation, but we have been unable to prove that. Note further that $\Z$ and $\Z^2$ are the only abelian one-relator groups \cite[Proposition 5.24, page 108]{LyndonSchupp}, and hence are the only abelian groups defined by redundant cyclic presentations.

\begin{theorem}\label{thm:defatleast2}
Let $n\geq 2$, $w \in F_n$ be a cyclically reduced word and assume that $P_n(w)$ is redundant. If either $n\geq 3$ or $n=2$ and $w$ is a proper power then $G_n(w)$ is large. In particular, if $P_n(w)$ is a redundant $(m,k,\nu)$-special cyclic presentation then $G_n(w)$ is large.
\end{theorem}

\begin{proof}
If $n\geq 3$ then Corollary \ref{cor:conciserefinementorientable} and Theorem \ref{thm:classificationofnonorientablepresentations} imply $\mathrm{def}(G)\geq 2$ so $G_n(w)$ is large by \cite{BaumslagPride}. If $n=2$ then $\mathrm{def}(G)\geq 1$, so if in addition $w$ is a proper power then $G_n(w)$ is large by \cite{Stohr}, \cite[page 83]{Gromov}. The final observation follows by noting that every $(m,k,\nu)$-special presentation has at least three generators.
\end{proof}

Therefore, if $P_n(w)$ is a redundant cyclic presentation that does not define a large group then (since there are no redundant cyclic presentations $P_1(w)$) $n=2$ and $w$ is not a proper power. By Theorems \ref{thm:classificationoforientableredundantpresentations28092021} and \ref{thm:classificationofnonorientablepresentations} a redundant cyclic presentation $P_2(w)$ in which $w$ is not a proper power is redundant if and only if $w=u(x_0,x_1)u(x_1,x_0)^\epsilon$, where $\epsilon=1$ in the orientable case and $\epsilon=-1$ in the non-orientable case, and so $G_2(w)$ is the one-relator group $\pres{x_0,x_1}{u(x_0,x_1)u(x_1,x_0)^\epsilon}$. One-relator groups either contain a non-abelian free subgroup or are solvable (\cite[Theorem 3]{KarrassSolitar}, \cite{Cebotar}) so we have the following:

\begin{corollary}[The Tits alternative]\label{cor:redundantTitsalternative}
If $P_n(w)$ is a redundant cyclic presentation then $G_n(w)$ either contains a non-abelian free subgroup or is solvable.
\end{corollary}

In the orientable case we can say precisely which presentations define a solvable group.

\begin{corollary}\label{cor:Titsaltorientable}
If $P_n(w)$ is an orientable redundant cyclic presentation then $G_n(w)$ contains a non-abelian free subgroup unless $n=2$ and either $w$ is a cyclic permutation of $x_0^{\epsilon}x_1^{\epsilon}$ ($\epsilon =\pm 1$), in which case $G\cong \Z$, or $w$ is a cyclic permutation of $x_0^{2\epsilon}x_1^{2\epsilon}$ ($\epsilon =\pm 1$), in which case $G\cong BS(1,-1)$.
\end{corollary}

\begin{proof}
As explained above, if $G$ does not contain a non-abelian free subgroup then $G\cong G_2(w)$ where $w$ is cyclically reduced and is a cyclic permutation of $u(x_0,x_1)u(x_1,x_0)$ for some reduced word $u(x_0,x_1)$.

If $u(x_0,x_1)$ involves exactly one of $x_0,x_1$ then $w$ is a cyclic permutation of $x_0^px_1^p$ for some $p\in \Z\backslash \{0\}$, so $G\cong \pres{x_0,x_1}{x_0^px_1^p}$ which is large if $|p|\geq 3$ (as it maps onto $\Z_p*\Z_p$), isomorphic to $BS(1,-1)$ if $p=\pm 2$, and isomorphic to $\Z$ if $p=\pm 1$. Thus we may assume that $u(x_0,x_1)$ involves both $x_0$ and $x_1$.

Without loss of generality we may write either (i) $u(x_0,x_1)=x_0^{k_1}x_1^{k_2}x_0^{k_3}\ldots x_1^{k_s}$, where $s\geq 2$ is even or (ii) $u(x_0,x_1)=x_0^{k_1}x_1^{k_2}x_0^{k_3}\ldots x_0^{k_s}$ where $s\geq 3$ is odd and (in each case) each $k_i\in \Z\backslash \{0\}$. If $s=2$ then $w$ is a cyclic permutation of $x_0^{k_1+k_2}x_1^{k_1+k_2}$ so $G\cong \pres{x_0,x_1}{x_0^{k_1+k_2}x_1^{k_1+k_2}}$, which maps onto $G_2(x_0^{k_1+k_2})\cong \Z_{|k_1+k_2|}*\Z_{|k_1+k_2|}$, which is large if $|k_1+k_2|\not \in \{1,2\}$, and $G$ is isomorphic to $\Z$ if $k_1+k_2=\pm 1$ and to $BS(1,-1)$ if $k_1+k_2=\pm 2$. Thus we may assume $s\geq 3$.

The shift extension of $G$ is of the form $E=\pres{x,t}{t^2, U(x,t)U(txt,t)}=\pres{x,t}{t^2, (U(x,t)t)^2}$ where
\[U(x,t) = \begin{cases}
  \left( \prod_{i=0}^{(s-3)/2} x^{k_{2i+1}}tx^{k_{2i+2}}t\right) x^{k_s} & \mathrm{if}~s~\mathrm{is~odd};\\
  \prod_{i=0}^{s/2-1} x^{k_{2i+1}}tx^{k_{2i+2}}t & \mathrm{if}~s~\mathrm{is~even}
\end{cases}\]
and hence $U(x,t)t$ is a cyclic permutation of $\prod_{i=1}^s x^{k_i}t$, if $s$ is odd, and to $x^{k_s+k_1}t\cdot \prod_{i=2}^{s-1} x^{k_i}t$, if $s$ is even. Note that $k_s+k_1\neq 0$, since $w$ is cyclically reduced. By \cite[Theorem 8]{FRRsurvey} or \cite[Lemma 7.3.3.1]{FRbook} $E$ contains a non-abelian free subgroup, and hence so does $G$, as required.
\end{proof}

To consider the remaining non-orientable cases we must consider the groups $G_2(u(x_0,x_1)u(x_1,x_0)^{-1})$.

\begin{lemma}
Let $u(x_0,x_1)$ be a reduced word and let $G=\pres{x_0,x_1}{u(x_0,x_1)u(x_1,x_0)^{-1}}$ where $u(x_0,x_1)=v(x_0,x_1)^t$ for some $t\geq 2$. Then $G$ contains a non-abelian free subgroup except if $t=2$ and either $v(x_0,x_1)=x_0^{\pm 1}$ or $x_1^{\pm 1}$, in which case $G$ is isomorphic to $BS(1,-1)$.
\end{lemma}

\begin{proof}
We may assume that $v$ is the root of $u$. The group $G$ maps onto $\pres{x_0,x_1}{v(x_0,x_1)^t,v(x_1,x_0)^t}$, which is large if $t>2$ by~\cite{EdjvetBalanced}, thus we may assume $t=2$. If $v(x_0,x_1)$ involves exactly one of $x_0,x_1$ then $G\cong \pres{x_0,x_1}{x_0^2x_1^{-2}}\cong BS(1,-1)$. Suppose then that $v(x_0,x_1)$ involves both $x_0$ and $x_1$. Without loss of generality we may write either (i) $v(x_0,x_1)=x_0^{k_1}x_1^{k_2}x_0^{k_3}\ldots x_1^{k_s}$, where $s\geq 2$ is even or (ii) $v(x_0,x_1)=x_0^{k_1}x_1^{k_2}x_0^{k_3}\ldots x_0^{k_s}$ where $s\geq 3$ is odd and (in each case) each $k_i\in \Z\backslash \{0\}$.

The group $G$ maps onto $H=\pres{x_0,x_1}{v(x_0,x_1)^2,v(x_1,x_0)^2}=G_2(v(x_0,x_1)^2)$. Then the shift extension of $H$ is of the form $E=\pres{x,t}{t^2, V(x,t)^2}$, where $V(x,t)=(x^{k_1}t)\ldots (x^{k_s}t)$ in case (i) (so $s\geq 2$) and $V(x,t)=(x^{k_s+k_1}t)\ldots (x^{k_{s-1}}t)$ in case (ii) (so $s-1\geq 2$). By \cite[Theorem 8]{FRRsurvey} or \cite[Lemma 7.3.3.1]{FRbook} $E$ contains a non-abelian free subgroup, and hence so does $H$, and hence $G$.
\end{proof}

Thus it remains to consider groups of the form $G=\pres{x_0,x_1}{u(x_0,x_1)u(x_1,x_0)^{-1}}$ where $u(x_0,x_1)$ is not a proper power. If $u(x_0,x_1)=x_0^{\pm 1}$ or $x_1^{\pm 1}$ then $G\cong \Z$; if $u(x_0,x_1)=(x_0x_1^{\pm 1})^{\pm 1}$ or $(x_1x_0^{\pm 1})^{\pm 1}$ then $G\cong \Z^2$. We expect that in all other cases $G$ contains a non-abelian free subgroup. As examples in this direction, (1) it follows from \cite[Corollary 3.4]{ButtonLargenessLERF} that if $G\not \cong \Z^2$ and $G^\mathrm{ab}\cong \Z^2$ then $G$ contains a non-abelian free subgroup and (2) that if $u(x_0,x_1)=x_0x_1x_0\ldots x_1x_0$ is of odd length $k\geq 3$ then the shift extension of $G$ is the group $E=\pres{x,t}{t^2, (xt)^k(x^{-1}t^{-1})^k}\cong \pres{y,t}{t^2, y^kty^{-k}t}$, which maps onto $\pres{y,t}{t^2,y^k}\cong \Z_2*\Z_k$, so is large. In summary, we pose:

\begin{conjecture}\label{conj:Titsalternative}
If $P_2(w)$ is a  non-orientable cyclic presentation then $G_2(w)$ either contains a non-abelian free subgroup or is isomorphic to $\Z,\Z^2$ or $BS(1,-1)$.
\end{conjecture}

\section{Star graphs of cyclic presentations}\label{sec:stargraphs}

In \cite[Theorem 3.3]{ChinyereWilliamsGP} the star graphs of concise cyclic presentations were described. In this section we generalize that result to describe the star graphs of (possibly redundant) cyclic presentations.

\subsection{Star graphs of orientable cyclic presentations}\label{sec:stargraphorientableredundant}

As in \cite{ChinyereWilliamsGP}, it is convenient to express results concerning star graphs in terms of multisets of differences of subscripts in length 2 subwords of a particular subword of $w$ (in the redundant case) or of length 2 cyclic subwords of $w$ (in the concise case).

\begin{defn}\label{def:ABQorientable}
Suppose $P_n(w)$ is orientable and let $u$ be the shortest  subword of $w$ such that $w=\prod_{i=0}^{n/(n,h)-1} \theta^{ih}(u)$ for any $0\leq h<n$. Let $\mathcal{A},\mathcal{B},\mathcal{Q},\mathcal{Q}^+,\mathcal{Q}^-$ be the multisets defined as follows:
\begin{alignat*}{1}
  \mathcal{A} &= \{ a\ |\ \mathrm{there~is~a~subword}~x_ix_{i+a}^{-1}~\mathrm{of}~u\theta^{h}(\iota(u)),~\mathrm{counting~multiplicities}\},\\
  \mathcal{B} &= \{ b\ |\ \mathrm{there~is~a~subword}~x_i^{-1}x_{i+b}~\mathrm{of}~u\theta^{h}(\iota(u)),~\mathrm{counting~multiplicities}\},\\
  \mathcal{Q} &= \{ q\ |\ \mathrm{there~is~a~subword}~x_ix_{i+q}~\mathrm{or}~x_{i+q}^{-1}x_{i}^{-1}~\mathrm{of}~u\theta^{h}(\iota(u)),~\mathrm{counting~multiplicities}\},\\
  \mathcal{Q}^+ &= \{ q\ |\ \mathrm{there~is~a~subword}~x_ix_{i+q}~\mathrm{of}~u\theta^{h}(\iota(u)),~\mathrm{counting~multiplicities}\},\\
  \mathcal{Q}^- &= \{ q\ |\ \mathrm{there~is~a~subword}~x_{i+q}^{-1}x_{i}^{-1}~\mathrm{of}~u\theta^{h}(\iota(u)),~\mathrm{counting~multiplicities}\}.
\end{alignat*}
\end{defn}

Note that $\mathcal{Q}=\emptyset$ if and only if $u$ is cyclically alternating and observe that the following congruence holds (compare \cite[(2)]{ChinyereWilliamsGP}):
\begin{alignat}{1}
\sum_{a\in \mathcal{A}} a + \sum_{b\in \mathcal{B}} b + \sum_{q\in \mathcal{Q^+}} q - \sum_{q\in \mathcal{Q^-}} q \equiv h\bmod n\label{eq:A+B+Q+-Q-}
\end{alignat}
and that if $P_n(w)$ is concise then the sets $\mathcal{A},\mathcal{B},\mathcal{Q},\mathcal{Q}^+,\mathcal{Q}^-$ are precisely the sets defined in \cite{ChinyereWilliamsGP}.

Theorem~\ref{thm:stargraphorientable} describes the star graph of a (redundant or concise) orientable cyclic presentation $P_n(w)$. Recall that the star graph of $P_n(w)$ is equal to the star graph of $P_n(v)$ where $v$ is the root of $w$, so we may assume that $w$ is not a proper power. The reader may find the following presentations helpful as illustrative examples: $P_6(x_0x_2^{-1}x_1^{-1}x_3x_5^{-1}x_4^{-1})$, $P_6(x_0x_3^{-1}x_1x_2^{-1}x_3x_0^{-1}x_4x_5^{-1})$.

\begin{theorem}\label{thm:stargraphorientable}
Suppose $P_n(w)$ is orientable where $w$ is not a proper power and let $u$ be the shortest subword of $w$ such that $w=\prod_{i=0}^{n/(n,h)-1} \theta^{i h}(u)$ for any $0\leq h<n$, and let $\Gamma$ be the star graph of $P_n(w)$. Let $d_\mathcal{A}=\mathrm{gcd}(n,a\ (a\in \mathcal{A}))$, $d_\mathcal{B}=\mathrm{gcd}(n,b\ (b\in \mathcal{B}))$, and if $u$ is not cyclically alternating let $q_0\in \mathcal{Q}$ and set $d=\mathrm{gcd}(n, a\ (a\in \mathcal{A}), b\ (b\in \mathcal{B}),q-q_0\ (q\in \mathcal{Q}))$. Then $\Gamma$ is $l(u)$-regular and has vertices $x_i,x_i^{-1}$ ($0\leq i<n$) and edges $x_i-x_{i+a}$, $x_i^{-1}-x_{i+b}^{-1}$, $x_i-x_{i+q}^{-1}$ for all $a\in \mathcal{A}, b\in \mathcal{B}, q\in\mathcal{Q}$, $0\leq i<n$.

\begin{itemize}
  \item[(a)] If $u$ is not cyclically alternating then $\Gamma$ has $d$ connected components $\Gamma_0,\ldots, \Gamma_{d-1}$ where for $0\leq j<d$ the graph $\Gamma_j$ is the induced subgraph of $\Gamma$ with vertex sets $V(\Gamma_j)=V(\Gamma_j^+)\cup V(\Gamma_j^-)$ where $\Gamma_j^+$ and $\Gamma_j^-$ are the induced subgraphs of $\Gamma$ with vertex sets
      \begin{alignat*}{1}
        V(\Gamma_j^+) &= \{ x_{j+td}\ |\ 0\leq t<n/d\},\\
        V(\Gamma_j^-) &= \{ x_{j+td+q_0}^{-1}\ |\ 0\leq t<n/d\}
      \end{alignat*}
      (subscripts mod~$n$). In particular, $|V(\Gamma_j^+)|=|V(\Gamma_j^-)|=n/d$ for all $0\leq j<d$ and the subscripts of the positive (resp.\,negative) vertices in any component are congruent mod $d$.

  \item[(b)] If $u$ is cyclically alternating then $\Gamma$ has $d_\mathcal{A}+d_\mathcal{B}$ connected components $\Gamma_0^+,\ldots, \Gamma_{d_\mathcal{A}-1}^+$, $\Gamma_0^-,\ldots, \Gamma_{d_\mathcal{B}-1}^-$,  which are, respectively, the induced labelled subgraphs of $\Gamma$ with vertex sets
      \begin{alignat*}{1}
        V(\Gamma_j^+) &= \{ x_{j+td_\mathcal{A}}\ |\ 0\leq t<n/d_\mathcal{A}\},\\
        V(\Gamma_j^-) &= \{ x_{j+td_\mathcal{B}}^{-1}\ |\ 0\leq t<n/d_\mathcal{B}\}
      \end{alignat*}
      (subscripts mod~$n$). Moreover, each component $\Gamma_j^+$ is isomorphic to the circulant graph \linebreak $\mathrm{circ}_{n/d_\mathcal{A}}(\{ a/d_\mathcal{A}\ (a\in\mathcal{A})\})$ and each component $\Gamma_j^-$ is isomorphic to circulant graph \linebreak $\mathrm{circ}_{n/d_\mathcal{B}}(\{ b/d_\mathcal{B}\ (b\in\mathcal{B})\})$.
\end{itemize}
\end{theorem}

\begin{proof}
By Corollary \ref{cor:conciserefinementorientable}, the star graph of $P_n(w)$ is equal to the star graph of the concise $(n,h)$-truncation $P_{n,(n,h)}(w)$ of $P_n(w)$. Let $P=P_{n,(n,h)}(w)$, $Q=P_n(u)$ with star graphs $\Gamma$ and $\Lambda$, respectively. Observe that each relator of $P$ has length $n/(n,h) \cdot l(u)$ so the sum of the lengths of the relators of $P$ is equal to $(n/(n,h) \cdot l(u))\cdot (n,h)=nl(u)$, which is the sum of the lengths of the relators of $Q$. By Corollary \ref{cor:conciserefinementorientable}, since $u$ is chosen to be the shortest possible, the presentation $Q$ is a concise cyclic presentation, and so its star graph is described in \cite[Theorem~3.3]{ChinyereWilliamsGP}. Our strategy is to identify the differences between $\Gamma$ and $\Lambda$. We use the multisets $\mathcal{A}, \mathcal{B}, \mathcal{Q}$ of Definition \ref{def:ABQorientable} for each presentation $P,Q$, denoting them $\mathcal{A}_P,\mathcal{B}_P,\mathcal{Q}_P$ (for $P$) and $\mathcal{A}_Q,\mathcal{B}_Q,\mathcal{Q}_Q$ (for $Q$). These multisets are related as follows (where $\cup$ and $\backslash$ denote multiset sum and multiset difference, respectively)
\begin{alignat*}{1}
  \mathcal{A}_Q &= \{ a\ |\ \mathrm{there~is~a~subword}~x_ix_{i+a}^{-1}~\mathrm{of}~u,~\mathrm{counting~multiplicities}\}\\
  &\qquad \cup \{ a\ |\ \mathrm{there~is~a~subword}~x_ix_{i+a}^{-1}~\mathrm{of}~\tau(u)\iota(u)\},\\
  \mathcal{A}_P
  &= \{ a\ |\ \mathrm{there~is~a~subword}~x_ix_{i+a}^{-1}~\mathrm{of}~u,~\mathrm{counting~multiplicities}\}\\
  &\qquad \cup \{ a\ |\ \mathrm{there~is~a~subword}~x_ix_{i+a}^{-1}~\mathrm{of}~\tau(u)\theta^{h}(\iota(u))\}\\
  &= (\mathcal{A}_Q \backslash \{ a\ |\ \mathrm{there~is~a~subword}~x_ix_{i+a}^{-1}~\mathrm{of}~\tau(u)\iota(u)\})\\
   &\qquad \cup \{ a\ |\ \mathrm{there~is~a~subword}~x_ix_{i+a}^{-1}~\mathrm{of}~\tau(u)\theta^{h}(\iota(u))\}.
  \intertext{Similarly,}
  \mathcal{B}_P&= (\mathcal{B}_Q \backslash \{ b\ |\ \mathrm{there~is~a~subword}~x_i^{-1}x_{i+b}~\mathrm{of}~\tau(u)\iota(u)\})\\
   &\qquad \cup \{ b\ |\ \mathrm{there~is~a~subword}~x_i^{-1}x_{i+b}~\mathrm{of}~\tau(u)\theta^{h}(\iota(u))\},\\
  \mathcal{Q}_P&= (\mathcal{Q}_Q \backslash
  \{ q\ |\ \mathrm{there~is~a~subword}~x_ix_{i+q}~\mathrm{or}~x_{i+q}^{-1}x_i^{-1}~\mathrm{of}~\tau(u)\iota(u)\}
  ) \\
  &\ \qquad \cup
  \{ q\ |\ \mathrm{there~is~a~subword}~x_ix_{i+q}~\mathrm{or}~x_{i+q}^{-1}x_i^{-1}~\mathrm{of}~\tau(u)\theta^{h}(\iota(u))\}.
\end{alignat*}

The graph $\Gamma$ is obtained from $\Lambda$ by, for each $0\leq i<n$, removing exactly one edge of the form $\theta^i(\tau(u)) - \theta^i(\iota(u))^{-1}$ and then adjoining one edge of the form $\theta^i(\tau(u)) - \theta^{i+h}(\iota(u))^{-1}$. By \cite[Theorem 3.3]{ChinyereWilliamsGP}, $\Lambda$ is $l(u)$-regular and has vertices $x_i,x_i^{-1}$ and edges $x_i-x_{i+a}$, $x_i^{-1}-x_{i+b}^{-1}$, $x_i-x_{i+q}^{-1}$ for all $a\in \mathcal{A}_Q$, $b\in \mathcal{B}_Q$, $q\in \mathcal{Q}_Q$, $0\leq i<n$. Therefore $\Gamma$ is also $l(u)$-regular, and has vertices $x_i,x_i^{-1}$ and the same (multi-)set of edges as $\Lambda$, but with the substitutions described above applied. That is, $\Gamma$ has edges $x_i-x_{i+a}$, $x_i^{-1}-x_{i+b}^{-1}$, $x_i-x_{i+q}^{-1}$ for all $a\in \mathcal{A}_P$, $b\in \mathcal{B}_P$, $q\in \mathcal{Q}_P$, $0\leq i<n$, completing the proof of first part of the statement. With this description of $\Gamma$ in place, the statements (a),(b) follow as in the proof of \cite[Theorem 3.3]{ChinyereWilliamsGP}.
\end{proof}

Analogously to \cite[Corollary 3.4]{ChinyereWilliamsGP} we have the following immediate corollary:

\begin{corollary}\label{cor:3knuand2knuorientable}
Let $P_n(w)$ be an orientable cyclic presentation in which $w$ is of length $k$ and is not a proper power and let $u$ be the shortest subword of $w$ such that $w=\prod_{i=0}^{n/(n,h)-1} \theta^{i h}(u)$ for any $0\leq h<n$. Then
\begin{itemize}
  \item[(a)] $P_n(w)$ is $(3,k,\nu)$-special if and only if $l(u)^2-l(u)+1=n/\nu$ and each component of the star graph is the incidence graph of a projective plane of order $l(u)-1$;
  \item[(b)] $P_n(w)$ is $(2,k,\nu)$-special if and only if $l(u)=n/\nu$ and each component of the star graph is the complete bipartite graph $K_{l(u),l(u)}$.
\end{itemize}
\end{corollary}

Analogously to \cite[Corollary 4.2]{ChinyereWilliamsGP} we have the following:

\begin{corollary}\label{cor:orientablegirthatmost8}
Let $P_n(w)$ be an orientable cyclic presentation in which $w$ is of length $k$ and is not a proper power and let $u$ be the shortest subword of $w$ such that $w=\prod_{i=0}^{n/(n,h)-1} \theta^{i h}(u)$ for any $0\leq h<n$. Suppose that $u$ has length at least 3 and let $\Gamma$ be the star graph of $P_n(w)$. If $u$ is a non-positive, non-negative, word of length $3$ then $\mathrm{girth}(\Gamma)\leq 8$; otherwise $\mathrm{girth}(\Gamma)\leq 6$.
\end{corollary}

\begin{proof}
If $l(u)\geq 4$ then by a modification of \cite[Lemma 4.1, Corollary 4.2]{ChinyereWilliamsGP} $\mathrm{girth}(\Gamma)\leq 6$. (More precisely, parts (a),(c),(d),(e),(f) of \cite[Lemma 4.1]{ChinyereWilliamsGP} hold for the star graph $\Gamma$ of the, possibly redundant, cyclic presentation $P_n(w)$ if ``$w$'' is replaced by ``$u$'' and ``cyclic subword'' is replaced by ``subword''; then the argument of \cite[Corollary 4.2]{ChinyereWilliamsGP} carries through as before, if ``cyclic subword'' is replaced by ``subword''.) Thus we may assume $l(u)=3$. By
inversion or cyclic permutation, and noting that by Lemma \ref{lem:allcycpermshavethisform} every cyclic permutation of $w$ has the form in the statement, we may assume that either (a) $u=x_0x_px_{p+q}$ or (b) $u=x_0x_px_{p+q}^{-1}$ for some $(0\leq p,q<n)$. In case (a) $\Gamma$ has edges $x_i-x_{i+p}^{-1}$, $x_i-x_{i+q}^{-1}$, $x_i-x_{i+h-p-q}^{-1}$ $(0\leq i<n)$. Then $\Gamma$ contains the reduced closed path $x_0-x_p^{-1}-x_{p-q}-x_{h-2q}^{-1}-x_{h-2q-p}-x_{h -q-p}^{-1}-x_0$ of length 6, so $\mathrm{girth}(\Gamma)\leq 6$. In case (b) $\Gamma$ has edges $x_i-x_{i+p}^{-1}$, $x_i-x_{i+q}$, $x_i^{-1}-x_{i+h-p-q}^{-1}$ $(0\leq i<n)$. Then $\Gamma$ contains the reduced closed path $x_0-x_p^{-1}-x_{h-q}^{-1}-x_{h-p-q}-x_{h-p}-x_{h}^{-1}-x_{p+q}^{-1}-x_q-x_0$ of length 8, so $\mathrm{girth}(\Gamma)\leq 8$.
\end{proof}

As in \cite[Theorem A]{ChinyereWilliamsGP}, it follows from \cite{HillPrideVella} that if $l(u)\geq 3$ and if $P_n(w)$ satisfies the small cancellation condition $T(q)$ where $q\geq 7$ then $l(u)=3$ and $u$ is non-positive and non-negative.

The following corollary is analogous to the final statement of \cite[Theorem A]{ChinyereWilliamsGP}.

\begin{corollary}\label{cor:orientable(>2,k,nu)}
Suppose $P_n(w)$ is orientable. If $P_n(w)$ is $(m,k,\nu)$-special with $m\geq 3$ then $m=3$ and $w$ is positive or negative.
\end{corollary}

\begin{proof}
Let $u$ be the shortest  subword of $w$ such that $w=\prod_{i=0}^{n/(n,h)-1} \theta^{i h}(u)$ for any $0\leq h< n$. If $l(u)<3$ then the star graph of $P_n(w)$ is at most 2-regular so $P_n(w)$ is not special, so we may assume $l(u)\geq 3$. If $u$ is positive or negative then by Corollary~\ref{cor:orientablegirthatmost8} $\mathrm{girth}(\Gamma)\leq 6$ so $m=3$, as required. Suppose then $u$ is non-positive and non-negative and that $P_n(w)$ is $(m,k,\nu)$-special, where $m\in \{3,4\}$ by Corollary~\ref{cor:orientablegirthatmost8}. Then there exists $a \in \mathcal{A}$. Then, letting $t=n/(n,a)$, $x_0-x_a-x_{2a}-\cdots  - x_{(t-1)a}-x_0$ is a reduced closed path of length $t$, which is therefore even (since $\Gamma$ is bipartite), so $t$ is even. Then (noting that $\nu$ divides $(n,a)$ by Theorem \ref{thm:stargraphorientable}) $n/\nu =t\cdot ((n,a)/\nu)$ is even. Since each component of $\Gamma$ is the incidence graph of a generalized $m$-gon it has 14 vertices if $m=3$, and $30$ vertices if $m=4$ (see \cite[Corollary 1.5.5]{Maldeghem}). Therefore $n/\nu=7$ or $15$, a contradiction.
\end{proof}

\subsection{Star graphs of non-orientable cyclic presentations}\label{sec:stargraphnonorientableredundantNEW}

We now define multisets of differences of subscripts in length 2 subwords of a particular cyclic subword of $w$ for the non-orientable case.

\begin{defn}\label{def:ABQnonorientable}
Suppose $w=u\theta^{n/2}(u)^{-1}$. Let $\bar{\mathcal{A}}',\bar{\mathcal{B}}',\bar{\mathcal{Q}}$ be the multisets defined as follows:
\begin{alignat*}{1}
  \bar{\mathcal{A}}' &= \{ a\ |\ \mathrm{there~is~a~subword}~x_ix_{i+a}^{-1}~\mathrm{of}~u,~\mathrm{counting~multiplicities}\},\\
  \bar{\mathcal{B}}' &= \{ b\ |\ \mathrm{there~is~a~subword}~x_i^{-1}x_{i+b}~\mathrm{of}~u,~\mathrm{counting~multiplicities}\},\\
  \bar{\mathcal{Q}} &= \{ q\ |\ \mathrm{there~is~a~subword}~x_ix_{i+q}~\mathrm{or}~x_{i+q}^{-1}x_{i}^{-1}~\mathrm{of}~u,~\mathrm{counting~multiplicities}\},
\end{alignat*}
and define $\bar{\mathcal{A}},\bar{\mathcal{B}}$ as follows:
\[ (\bar{\mathcal{A}},\bar{\mathcal{B}} )=
\begin{cases}
(\bar{\mathcal{A}}' \cup \{n/2\},  \bar{\mathcal{B}}' \cup \{n/2\}) & \iota(u),\tau(u)~\text{both~positive~or~both~negative},\\
(\bar{\mathcal{A}}',  \bar{\mathcal{B}}' \cup \{n/2,n/2\})& \iota(u)~\text{positive~and}~\tau(u)~\text{negative},\\
(\bar{\mathcal{A}}'\cup \{n/2,n/2\},  \bar{\mathcal{B}}') & \iota(u)~\text{negative~and}~\tau(u)~\text{positive}.
\end{cases}
\]
\end{defn}

Note that $\bar{\mathcal{Q}}=\emptyset$ if and only if $u$ is alternating and that, letting $v$ denote the cyclic subword $\theta^{n/2}(\iota(u)^{-1}) u \theta^{n/2} (\tau (u))^{-1}$ of $w$,
\begin{alignat*}{1}
  \bar{\mathcal{A}} &= \{ a\ |\ \mathrm{there~is~a~subword}~x_ix_{i+a}^{-1}~\mathrm{of}~v,~\mathrm{counting~multiplicities}\},\\
  \bar{\mathcal{B}} &= \{ b\ |\ \mathrm{there~is~a~subword}~x_i^{-1}x_{i+b}~\mathrm{of}~v,~\mathrm{counting~multiplicities}\}.
\end{alignat*}
Theorem~\ref{thm:stargraphnonorientableNEW} describes the star graph of a non-orientable cyclic presentation $P_n(u\theta^{n/2}(u^{-1}))$. We may assume that $u\theta^{n/2}(u^{-1})$ is not a proper power (noting that proper power words of this form exist, such as $(x_0x_1^{-1})^2$ when $n=2$). The reader may find the following presentations helpful as illustrative examples: $P_4(x_0x_1x_3^{-1}x_2^{-1})$, $P_4(x_1x_3^{-1}x_2^{-1}x_0x_2^{-1}x_0x_1x_3^{-1})$, $P_4(x_0x_1^{-1}x_3x_2^{-1})$,
$P_4(x_0x_1^{-1}x_0 x_2^{-1}x_3x_2^{-1})$.

\begin{theorem}\label{thm:stargraphnonorientableNEW}
Let $\Gamma$ be the star graph of $P_n(w)$ where $w=u\theta^{n/2}(u)^{-1}$ and is not a proper power. Let $d_{\bar{\mathcal{A}}}=\mathrm{gcd}(n,a\ (a\in \bar{\mathcal{A}}))$, $d_{\bar{\mathcal{B}}}=\mathrm{gcd}(n,b\ (b\in \bar{\mathcal{B}}))$, and if $u$ is not alternating let $q_0\in \bar{\mathcal{Q}}$ and set $d=\mathrm{gcd}(n, a\ (a\in \bar{\mathcal{A}}), b\ (b\in \bar{\mathcal{B}}), q-q_0\ (q\in \bar{\mathcal{Q}}))$. Then $\Gamma$ is $l(u)$-regular and has vertices $x_i,x_i^{-1}$ ($0\leq i<n$) and edges $x_i-x_{i+a}$, $x_i^{-1}-x_{i+b}^{-1}$, $x_i-x_{i+q}^{-1}$ for all $a \in \bar{\mathcal{A}}'$, $b\in \bar{\mathcal{B}}'$, $q\in \bar{\mathcal{Q}}'$, $0\leq i<n$ and edges
$x_i-x_{i+a}$, $x_i^{-1}-x_{i+b}^{-1}$, for all $a \in \bar{\mathcal{A}}\backslash \bar{\mathcal{A}}'$, $b\in \bar{\mathcal{B}}\backslash \bar{\mathcal{B}}'$, $0\leq i<n/2$.

\begin{itemize}
  \item[(a)] If $u$ is not alternating then $\Gamma$ has $d$ connected components $\Gamma_0,\ldots, \Gamma_{d-1}$ where for $0\leq j<d$ the graph $\Gamma_j$ is the induced subgraph of $\Gamma$ with vertex sets $V(\Gamma_j)=V(\Gamma_j^+)\cup V(\Gamma_j^-)$ where $\Gamma_j^+$ and $\Gamma_j^-$ are the induced subgraphs of $\Gamma$ with vertex sets
      \begin{alignat*}{1}
        V(\Gamma_j^+) &= \{ x_{j+td}\ |\ 0\leq t<n/d\},\\
        V(\Gamma_j^-) &= \{ x_{j+td+q_0}^{-1}\ |\ 0\leq t<n/d\}
      \end{alignat*}
      (subscripts mod~$n$). In particular, $|V(\Gamma_j^+)|=|V(\Gamma_j^-)|=n/d$ for all $0\leq j<d$ and the subscripts of the positive (resp.\,negative) vertices in any component are congruent mod $d$.

  \item[(b)] If $u$ is alternating then $\Gamma$ has $d_{\bar{\mathcal{A}}}+d_{\bar{\mathcal{B}}}$ connected components $\Gamma_0^+,\ldots, \Gamma_{d_{\bar{\mathcal{A}}}-1}^+$, $\Gamma_0^-,\ldots, \Gamma_{d_{\bar{\mathcal{B}}}-1}^-$,  which are, respectively, the induced labelled subgraphs of $\Gamma$ with vertex sets
      \begin{alignat*}{1}
        V(\Gamma_j^+) &= \{ x_{j+td_{\bar{\mathcal{A}}}}\ |\ 0\leq t<n/d_{\bar{\mathcal{A}}}\},\\
        V(\Gamma_j^-) &= \{ x_{j+td_{\bar{\mathcal{B}}}}^{-1}\ |\ 0\leq t<n/d_{\bar{\mathcal{B}}}\}
      \end{alignat*}
      (subscripts mod~$n$). Moreover, each component $\Gamma_j^+$ is isomorphic to the graph \linebreak $\mathrm{circ}'_{n/d_{\bar{\mathcal{A}}}}(\{ a/d_{\bar{\mathcal{A}}}\ (a\in\bar{\mathcal{A}})\})$ and each component $\Gamma_j^-$ is isomorphic to $\mathrm{circ}'_{n/d_{\bar{\mathcal{B}}}}(\{ b/d_{\bar{\mathcal{B}}}\ (b\in\bar{\mathcal{B}})\})$.
\end{itemize}
\end{theorem}

\begin{proof}
By Theorem \ref{thm:classificationofnonorientablepresentations} the star graph of $P_n(w)$ is equal to the star graph of $P_{n,n/2}(w)$, which is concise.
Therefore every length two cyclic subword of each relator $\theta^{i}(w)$ ($0\leq i<n/2$) contributes exactly one edge to the star graph $\Gamma$.

Noting that $w$ is of the form $\iota(u)\ldots \tau(u) \cdot \theta^{n/2}(\tau(u)^{-1}) \ldots \theta^{n/2}(\iota (u)^{-1})$, the length two cyclic subwords of the $n/2$ relators $\theta^i(w)$ ($0\leq i<n/2$) consist of the $l(u)-1$ length two subwords of $\theta^i(u)$, the $l(u)-1$ length two subwords of $\theta^{i+n/2}(u^{-1})$, $\theta^i(\tau(u)) \cdot \theta^{i+n/2}(\tau(u)^{-1})$, and of $\theta^{i+n/2}(\iota(u)^{-1}) \cdot \theta^{i}(\iota(u))$.
These contribute edges
$x_i-x_{i+a}$ for each $a \in \bar{\mathcal{A}}'$, $x_i^{-1}-x_{i+b}^{-1}$ for each $b \in \bar{\mathcal{B}}'$, $x_i-x_{i+q}^{-1}$ for each $q \in \bar{\mathcal{Q}}$ ($0\leq i<n$), $\theta^i(\tau(u)) - \theta^{i+n/2}(\tau(u))$, $\theta^{i+n/2}(\iota(u)^{-1}) - \theta^{i}(\iota(u)^{-1})$ ($0\leq i<n/2$).
Equivalently, they contribute the edges $x_i-x_{i+a}$ for each $a \in \bar{\mathcal{A}}'$, $x_i^{-1}-x_{i+b}^{-1}$ for each $b \in \bar{\mathcal{B}}'$, $x_i-x_{i+q}^{-1}$ for each $q \in \bar{\mathcal{Q}}$ ($0\leq i<n$), and the edges $x_i-x_{i+a}$ ($a\in \bar{\mathcal{A}}\backslash \bar{\mathcal{A}}'$), $x_i^{-1}-x_{i+b}^{-1}$ ($b\in \bar{\mathcal{B}}\backslash \bar{\mathcal{B}}'$), ($0\leq i<n/2$).

Thus each positive vertex has the same degree and since in the star graph of any finite presentation vertices corresponding to a generator and its inverse have the same degree (see \cite[Section 2.3.3]{HillPrideVella}), the graph $\Gamma$ is regular. Moreover, the number of edges of the star graph of a concise presentation that has no proper power relators, and where the relators are cyclically reduced, is equal to the sum of the lengths of the relators, so the number of edges is equal to $nl(u)$, and hence $\Gamma$ is $l(u)$-regular.

This completes the proof of first part of the statement. With this description of $\Gamma$ in place, the statements (a),(b) follow as in the proof of \cite[Theorem 3.3]{ChinyereWilliamsGP}.
\end{proof}

\begin{corollary}\label{cor:nonorientablegirthatmost4}
Let $\Gamma$ be the star graph of $P_n(u\theta^{n/2}(u)^{-1})$ where $l(u)\geq 2$ and $u$ is reduced, and let $\epsilon_{\iota},\epsilon_{\tau}\in \{1,-1\}$ be the exponents of $\iota(u),\tau(u)$, respectively. Then $\mathrm{girth}(\Gamma)\leq 4$ and if $\epsilon_{\iota}\epsilon_{\tau}=-1$ then $\mathrm{girth}(\Gamma)=2$.
\end{corollary}

\begin{proof}
If $\epsilon_\iota=1, \epsilon_\tau=-1$ (resp.\,$\epsilon_\iota=-1, \epsilon_\tau=1$) then $x_0-x_{n/2}-x_0$ (resp.\,$x_0^{-1}-x_{n/2}^{-1}-x_0^{-1}$) is a closed path of length 2. If $\bar{\mathcal{Q}}\neq\emptyset$ then there exists $q\in \bar{\mathcal{Q}}$ so $x_0-x_{n/2}-x_{n/2+q}^{-1}-x_q^{-1}-x_0$ is a closed path of length 4. Thus we may assume $\bar{\mathcal{Q}}=\emptyset$ and $\epsilon_\iota=\epsilon_\tau$; that is, $u$ is alternating but not cyclically alternating. Then $|\bar{\mathcal{A}}|\geq 2$, $|\bar{\mathcal{B}}|\geq 2$ so  $\mathrm{girth} (\Gamma^+)\leq 4$ and $\mathrm{girth} (\Gamma^-)\leq 4$ so $\mathrm{girth} (\Gamma)\leq 4$.
\end{proof}

As in \cite[Theorem A]{ChinyereWilliamsGP}, it follows from \cite{HillPrideVella} that if $l(u)\geq 2$ then $P_n(u\theta^{n/2}(u)^{-1})$ does not satisfy the small cancellation condition $T(5)$.

\section{$(m,k,\nu)$-special cyclic presentations}\label{sec:specialpresentationsclassification}

The concise $(m,k,\nu)$-special cyclic presentations were classified in \cite[Sections 5,6]{ChinyereWilliamsGP}. In this section we generalize that classification to (possibly redundant) cyclic presentations. By Corollary \ref{cor:orientable(>2,k,nu)} and Corollary~\ref{cor:nonorientablegirthatmost4}, if a cyclic presentation $P_n(w)$ is $(m,k,\nu)$-special where $m\geq 3$ then $P_n(w)$ is orientable, $w$ is positive or negative, and $m=3$. In Section~\ref{sec:3knuspecialpresentations} (Theorem~\ref{thm:3knuspecial}) we classify the (redundant and concise) orientable $(3,k,\nu)$-special cyclic presentations $P_n(w)$ where $w$ is positive. In Section \ref{sec:orientable2knuspecialpresentations} (Theorems \ref{thm:2knupositive} --\ref{thm:2knunonpositivenonnegativenonalternating}) we classify the (redundant and concise) orientable $(2,k,\nu)$-special cyclic presentations. In Section \ref{sec:nonorientable2knuspecialpresentations} (Theorems~\ref{thm:2knunonoralt}, \ref{thm:2knunonornonalt}) we classify the non-orientable $(2,k,\nu)$-special cyclic presentations. Theorems \ref{thm:2knupositive}--\ref{thm:2knunonornonalt} are analogous to Theorems C,D,E of \cite{ChinyereWilliamsGP}, which deal with concise cyclic presentations. The proofs of those theorems proceed by analysing graphs defined in terms of sets $\mathcal{A},\mathcal{B},\mathcal{Q},\mathcal{Q}^+,\mathcal{Q}^-$. Theorems \ref{thm:3knuspecial}--\ref{thm:2knunonpositivenonnegativenonalternating} (which deal with the orientable case) are re-expressions of the corresponding theorems from \cite{ChinyereWilliamsGP} in terms of the sets $\mathcal{A},\mathcal{B},\mathcal{Q},\mathcal{Q}^+,\mathcal{Q}^-$ of Definition \ref{def:ABQorientable}. (Note that these statements require $l(u)\geq 3$ to ensure that the degrees of vertices are at least 3, as required by Definition \ref{def:mkalphaspecialpres}.) The proofs are similarly analogous to those from \cite{ChinyereWilliamsGP}, so are omitted. A minor exception to this is Theorem \ref{thm:2knupositive}, where in \cite[Theorem C]{ChinyereWilliamsGP} the `non-redundant' (i.e.\,concise) hypothesis provides the further conclusion that $\nu\leq 2$, a condition that does not hold in the redundant case. Theorems \ref{thm:2knunonoralt} and \ref{thm:2knunonornonalt} (which deal with the non-orientable case) and their proofs are new.

It will follow from the results of Sections \ref{sec:3knuspecialpresentations}--\ref{sec:nonorientable2knuspecialpresentations}, together with Theorem \ref{thm:classificationoforientableredundantpresentations28092021}, that if $P_n(w)$ is a redundant $(m,k,\nu)$-special cyclic presentation then $k\geq 6$, so $2/k+1/m<1$ and hence (as described in Section~\ref{sec:specialpresentation}) the corresponding group $G_n(w)$ is non-elementary hyperbolic.

Example~\ref{ex:redundantspecial} exhibits the various situations that arise for redundant $(m,k,\nu)$-special cyclic presentations.

\begin{example}\label{ex:redundantspecial}
\begin{itemize}
  \item[(a)]
      \begin{itemize}
        \item[(i)] $P_{7}(\prod_{i=0}^6 \theta^{4i} (x_0^2x_1))$ is orientable, redundant, $(3,21,1)$-special, and defines the one-relator group $G_{7,1}(\prod_{i=0}^6 \theta^{4i} (x_0^2x_1))$;
        \item[(ii)] $P_{14}(x_0x_1x_{10}x_7x_8x_3)$ is orientable, redundant, and $(3,6,2)$-special;
        \item[(iii)] $P_{21}(\prod_{i=0}^{6}\theta^{3i}(x_0x_{2}x_7))$ is orientable, redundant, and $(3,21,3)$-special.
      \end{itemize}
      In each case, each component of the star graph is the Heawood graph.

  \item[(b)]  $P_9(x_0x_1x_5x_3x_4x_8x_6x_7x_2)$ is orientable, redundant, and $(2,9,3)$-special; \linebreak $P_8(x_0x_1x_3x_6x_2x_3x_5x_0x_4x_5x_7x_2x_6x_7x_1x_4)$ is orientable, redundant, and $(2, 16, 1)$-special. In these cases the defining word is positive.

  \item[(c)] $P_8(x_0x_1^{-1} x_6x_3^{-1} x_4x_5^{-1}  x_2x_7^{-1})$ is orientable, redundant and $(2, 8, 2)$-special, and the defining word is cyclically alternating.

  \item[(d)]  $P_8(x_0x^{-1}_2 x_4x_7x_6x^{-1}_0 x_2x_5x_4x^{-1}_6 x_0x_3x_2x^{-1}_4 x_6x_1)$ is orientable, redundant, and $(2, 16, 2)$-special, and the defining word is non-positive, non-negative, and  non-alternating.

  \item[(e)]  $P_6(x_0x_1^{-1}x_0x_3^{-1}x_4x_3^{-1})$ is non-orientable and $(2,6,2)$-special, and the defining word is cyclically alternating.

  \item[(f)]  $P_{12}(x_0x_2^{-1}x_4x_7x_2x_1 x_7^{-1}x_8^{-1}x_1^{-1}x_{10}^{-1}x_8x_6^{-1})$ is non-orientable and $(2,12,2)$-special and the defining word is non-alternating.
\end{itemize}
\end{example}

\subsection{$(3,k,\nu)$-special cyclic presentations}\label{sec:3knuspecialpresentations}

Theorem \ref{thm:3knuspecial} classifies when a (possibly redundant) orientable cyclic presentation $P_n(w)$ in which $w$ is a positive word, is $(3,k,\nu)$-special and so generalises \cite[Theorem B]{ChinyereWilliamsGP}; its proof is analogous to the proof of that theorem. See Example~\ref{ex:redundantspecial}(a) for examples.

\begin{theorem}\label{thm:3knuspecial}
Suppose $P_n(w)$ is an irreducible, orientable cyclic presentation, where $w$ is a positive word of length $k\geq 3$ that is not a proper power, and let $u$ be the shortest subword of $w$ such that $w=\prod_{i=0}^{n/(n,h)-1} \theta^{i h}(u)$ for any $0\leq h<n$. Then $P_n(w)$ is $(3,k,\nu)$-special if and only if the following hold:
\begin{itemize}
\item[(a)] $n=\nu N$ where $N=l(u)^2-l(u)+1$; and
\item[(b)] $\mathcal{Q}$ is a perfect difference set; and
\item[(c)] $q\equiv q' \bmod \nu$ for each pair $q,q'\in \mathcal{Q}$.%; and
\end{itemize}
\end{theorem}

We remark that a consequence of conditions (a),(c) of Theorem \ref{thm:3knuspecial} is that $\nu$ divides $k$.

\subsection{Orientable $(2,k,\nu)$-special cyclic presentations}\label{sec:orientable2knuspecialpresentations}

In this section we classify the orientable $(2,k,\nu)$-special cyclic presentations $P_n\left( \prod_{i=0}^{n/(n,h)-1} \theta^{i h}(u)\right)$ (where we may assume that $u$ is the shortest possible). We consider the cases $u$ positive, cyclically alternating, and non-positive, non-negative, and not cyclically alternating separately.

\subsubsection{$u$ positive}\label{sec:positivew}

Theorem \ref{thm:2knupositive} classifies when a (redundant or concise) orientable cyclic presentation $P_n\left( \prod_{i=0}^{n/(n,h)-1} \theta^{i h}(u)\right)$ in which $u$ is a positive word, is $(2,k,\nu)$-special
and so generalises \cite[Theorem C]{ChinyereWilliamsGP}. Note that, unlike in \cite[Theorem C]{ChinyereWilliamsGP} $\nu$ is no longer limited to the values $1,2$. The proof of Theorem \ref{thm:2knupositive} is analogous to that of \cite[Theorem C]{ChinyereWilliamsGP} except that the argument that leads to the conclusion $\nu\in \{1,2\}$ needs to be removed. See Example~\ref{ex:redundantspecial}(b) for an example.

\begin{theorem}\label{thm:2knupositive}
Let $P_n(w)$ be an irreducible, orientable cyclic presentation, where $w$ has length $k\geq 4$ that is not a proper power, and let $u$ be the shortest subword of $w$ such that $w=\prod_{i=0}^{n/(n,h)-1} \theta^{i h}(u)$ for any $0\leq h<n$, and suppose that $u$ is positive. Then $P_n(w)$ is $(2,k,\nu)$-special if and only if $n=\nu l(u)$, $l(u)\geq 3$, $\mathcal{Q}=\{q_0,\nu+q_0,2\nu+q_0,\ldots , (l(u)-1)\nu+q_0\}$ for some $0\leq q_0<n$ such that $\mathrm{gcd}(q_0,\nu)=1$, in which case each component of $\Gamma$ is isomorphic to $K_{l(u),l(u)}$.
\end{theorem}

In constructing examples it is useful to note that by (\ref{eq:A+B+Q+-Q-}) the conditions of Theorem \ref{thm:2knupositive} imply $l(u)q_0+[l(u)(l(u)-1)/2]\nu \equiv h\bmod n$.

\subsubsection{$u$ cyclically alternating}\label{sec:alternatingw}

Theorem \ref{thm:2knualternating}  classifies when a (redundant or concise) orientable cyclic presentation $P_n\left( \prod_{i=0}^{n/(n,h)-1} \theta^{i h}(u)\right)$ in which $u$ is a cyclically alternating word, is $(2,k,\nu)$-special and so generalises \cite[Theorem D]{ChinyereWilliamsGP}; its proof is analogous to the proof of that theorem. See Example~\ref{ex:redundantspecial}(c) for an example.

\begin{theorem}\label{thm:2knualternating}
Let $P_n(w)$ be an irreducible, orientable cyclic presentation, where $w$ has length $k\geq 4$ and is not a proper power, and let $u$ be the shortest subword of $w$ such that $w=\prod_{i=0}^{n/(n,h)-1} \theta^{i h}(u)$ for any $0\leq h<n$, and suppose that $u$ is cyclically alternating. Then $P_n(w)$ is $(2,k,\nu)$-special if and only if $n=2l(u)$, $l(u)\geq 3$, $\nu=2$, and $\mathcal{A},\mathcal{B}$ are each sets of the form $\{ \pm 1, \pm 3, \ldots ,\pm (n/2-1)\}$.
\end{theorem}

\subsubsection{$u$ non-positive, non-negative, not cyclically alternating}\label{sec:nonposnonnegnonaltw}

Theorem \ref{thm:2knunonpositivenonnegativenonalternating}  classifies when a (redundant or concise) orientable cyclic presentation presentation \linebreak $P_n\left( \prod_{i=0}^{n/(n,h)-1} \theta^{i h}(u)\right)$ in which $u$ is a non-positive, non-negative, and not cyclically alternating word, is $(2,k,\nu)$-special and so generalises \cite[Theorem E]{ChinyereWilliamsGP}; its proof is analogous to the proof of that theorem. See Example~\ref{ex:redundantspecial}(d) for an example.

\begin{theorem}\label{thm:2knunonpositivenonnegativenonalternating}
Let $P_n(w)$ be an irreducible, orientable cyclic presentation, where $w$ has length $k\geq 4$ and is not a proper power, and let $u$ be the shortest subword of $w$ such that $w=\prod_{i=0}^{n/(n,h)-1} \theta^{i h}(u)$ for any $0\leq h<n$, and suppose that $u$ is non-positive, non-negative, and not cyclically alternating. Then $P_n(w)$ is $(2,k,\nu)$-special if and only if the following hold:
\begin{itemize}
  \item[(a)] $n=\nu l(u)$ and $l(u)$ is divisible by 4;
  \item[(b)] $\mathcal{A},\mathcal{B}$ are each sets of the form $\{\pm \nu, \pm 3\nu, \ldots , \pm ((n/2\nu)-1)\nu \}$;
  \item[(c)] $\mathcal{Q}^+\cap\mathcal{Q}^-=\emptyset$ and there exists some $0\leq q_0<n$ with $(q_0,\nu)=1$ such that $\mathcal{Q}= \{ q_0,q_0+2\nu,\ldots , q_0+(n/\nu-2)\nu\}$.
\end{itemize}
\end{theorem}

\subsection{Non-orientable $(2,k,\nu)$-special cyclic presentations}\label{sec:nonorientable2knuspecialpresentations}

In this section we classify the $(2,k,\nu)$-special cyclic presentations $P_n(w)$ where $w=u\theta^{n/2}(u)^{-1}$.

\subsubsection{$w$ alternating}\label{sec:alternatingu}

Theorem \ref{thm:2knunonoralt} classifies when a cyclic presentation $P_n(w)$, where $w=u\theta^{n/2}(u)^{-1}$ is an alternating word, is $(2,k,\nu)$-special and so extends \cite[Theorem D]{ChinyereWilliamsGP}.  See Example~\ref{ex:redundantspecial}(e) for an example.

\begin{theorem}\label{thm:2knunonoralt}
Let $w=u\theta^{n/2}(u)^{-1}$ be a word of length $k$ that is alternating and not a proper power and suppose that $P_n(w)$ is irreducible. Then $P_n(w)$ is $(2,k,\nu)$-special if and only if  $l(u)=n/2\geq 3$ is odd, $\nu=2$, and $\bar{\mathcal{A}},\bar{\mathcal{B}}$ are each sets of the form $\{ \pm 1, \pm 3, \ldots , \pm (n-4)/2, n/2\}$.
\end{theorem}

\begin{proof}
Let $\Gamma$ be the star graph of $P_n(w)$ and let $\Gamma^+,\Gamma^-$ be the induced subgraphs of $\Gamma$ whose vertices are the positive and negative vertices of $\Gamma$, respectively, and let $d_{\bar{\mathcal{A}}}=\mathrm{gcd} (n, a\ (a \in \bar{\mathcal{A}}))$, $d_{\bar{\mathcal{B}}}=\mathrm{gcd} (n, b\ (b \in \bar{\mathcal{B}}))$. By Theorem \ref{thm:stargraphnonorientableNEW} $\Gamma^+$ has $d_{\bar{\mathcal{A}}}$ components and $\Gamma^-$ has $d_{\bar{\mathcal{B}}}$ components. In this proof we use properties of $\Gamma$ provided by Theorem \ref{thm:stargraphnonorientableNEW} freely without further reference.

Suppose first that $P_n(w)$ is $(2,k,\nu)$-special. By Corollary \ref{cor:nonorientablegirthatmost4} $u$ is not cyclically alternating so $l(u)$ is odd. Then each component of $\Gamma$ is isomorphic, so $d_{\bar{\mathcal{A}}}=d_{\bar{\mathcal{B}}}$, and since $P_n(w)$ is irreducible $1=(d_{\bar{\mathcal{A}}},d_{\bar{\mathcal{B}}})=d_{\bar{\mathcal{A}}}=d_{\bar{\mathcal{B}}}$, so $\Gamma$ has 2 components, so $\nu=2$. Hence $\Gamma^+,\Gamma^-$ are each isomorphic to $K_{n/2,n/2}$, which is $n/2$-regular, so $l(u)=n/2$ and (by the definition of special presentations) $l(u)\geq 3$. Moreover, $K_{n/2,n/2}=\Gamma^+=\mathrm{circ}'_n(\bar{\mathcal{A}})$ and $K_{n/2,n/2}=\Gamma^-=\mathrm{circ}'_n(\bar{\mathcal{B}})$ so $\bar{\mathcal{A}},\bar{\mathcal{B}}$ are each sets of the form $\{\pm 1, \pm 3 , \ldots , \pm (n-4)/2 , n/2\}$, as required.

Conversely, suppose that the given conditions holds. Then $d_{\bar{\mathcal{A}}}=1$ and the edges of $\Gamma^+$ join each even vertex $x_i$ to each odd vertex $x_j$, and so $\Gamma^+$ is the complete bipartite graph $K_{n/2,n/2}$. Similarly $\Gamma^-$ is the complete bipartite $K_{n/2,n/2}$. Since the degree of each vertex is $l(u)\geq 3$, the presentation $P_n(w)$ is $(2,k,\nu)$-special.
\end{proof}

\subsubsection{$w$ non-alternating}\label{sec:nonalternatingu}

Theorem \ref{thm:2knunonornonalt}  classifies when a cyclic presentation $P_n(u\theta^{n/2}(u)^{-1})$ in which $u$ is a non-alternating word, is $(2,k,\nu)$-special and so extends \cite[Theorem E]{ChinyereWilliamsGP}. See Example~\ref{ex:redundantspecial}(f) for an example.

\begin{theorem}\label{thm:2knunonornonalt}
Let $w=u\theta^{n/2}(u)^{-1}$ be a word of length $k$ that is not a proper power, where $u$ is a non-alternating word of length at least 3, and suppose that $P_n(w)$ is irreducible. Then $P_n(w)$ is $(2,k,\nu)$-special if and only if the following hold:
\begin{itemize}
  \item[(a)] $n=\nu l(u)$, $l(u)\equiv 2 \bmod 4$, $l(u)\geq 6$;

  \item[(b)] $\bar{\mathcal{A}},\bar{\mathcal{B}}$ are each sets of the form $\{ \pm \nu,\pm 3\nu, \ldots , \pm (n/(2\nu)-2)\nu , n/2\}$;

  \item[(c)] there exists some $0<q_0<n$ with $(q_0,\nu)=1$ such that $\bar{\mathcal{Q}}=\{q_0,q_0+2\nu,\ldots, q_0+(n/\nu-2)\nu\}$.
\end{itemize}
\end{theorem}

\begin{proof}
Suppose first that $P_n(w)$ is $(2,k,\nu)$-special. Then $\bar{\mathcal{A}}, \bar{\mathcal{B}},\bar{\mathcal{Q}}$ are sets (i.e. they have no repeated elements) and by Corollary \ref{cor:nonorientablegirthatmost4} $\bar{\mathcal{A}}\backslash \bar{\mathcal{A}}'=\{n/2\}$ and $\bar{\mathcal{B}}\backslash \bar{\mathcal{B}}'=\{n/2\}$.

By Theorem \ref{thm:stargraphnonorientableNEW}, each component $\Gamma_i$ of the star graph $\Gamma$ of $P_n(w)$ is the complete bipartite graph $K_{l(u),l(u)}$ so $n=\nu l(u)$. Moreover, $\nu <n/2$, and hence $l(u)>2$, for otherwise each component has at most four vertices, so the vertices have degree at most 2, a contradiction. We use properties of $\Gamma$ provided by Theorem \ref{thm:stargraphnonorientableNEW} freely without further reference. The component $\Gamma_0$ has vertex set
\( V(\Gamma_0)=V(\Gamma_0^+) \cup V(\Gamma_0^-) \)
where  $\Gamma_0^+, \Gamma_0^-$ are the induced labelled subgraphs of $\Gamma_0$ with vertex sets
\[ V(\Gamma_0^+)= \{ x_0,x_\nu,\ldots , x_{(n/\nu-1)\nu}\}\]
and
\[ V(\Gamma_0^-)= \{ x_{q_0}^{-1},x_{q_0+\nu}^{-1},\ldots , x_{q_0+(n/\nu-1)\nu}^{-1}\}\]
for some $q_0\in \bar{\mathcal{Q}}$ (which is non-empty, since $w$ is non-alternating). In particular $\nu | a$ for all $a \in \bar{\mathcal{A}}'$ and $\nu | b$ for all $b \in \bar{\mathcal{B}}'$.

Suppose for contradiction that $\nu,n-\nu \not \in \bar{\mathcal{A}}$. Then for each $0\leq i<n$ vertices $x_i,x_{i+\nu}$ are not joined by an edge. Therefore the positive vertices of $\Gamma_0$ are all in the same part of $\Gamma_0$, so $\bar{\mathcal{A}}=\emptyset$. By Corollary \ref{cor:nonorientablegirthatmost4} $\iota(u),\tau(u)$ are both positive or both negative, so $n/2\in \bar{\mathcal{A}}$, a contradiction, so $\nu$ or $n-\nu\in \bar{\mathcal{A}}$. Similarly $\nu$ or $n-\nu\in \bar{\mathcal{B}}$. Therefore $\Gamma_0$ contains closed paths $x_0-x_\nu-\cdots - x_{(n/\nu-1)\nu}-x_0$ and $x_{q_0}^{-1}-x_{q_0+\nu}^{-1}-\cdots - x_{q_0+(n/\nu-1)\nu}^{-1}-x_{q_0}^{-1}$, each of length $n/\nu$, which is therefore even, since $\Gamma$ is bipartite. Therefore the vertices $x_\nu,x_{3\nu},\ldots ,x_{(n/\nu-1)\nu}$ are precisely those positive vertices of $\Gamma_0$ that belong to a different part of $\Gamma_0$ to $x_0$ (and so are neighbours of $x_0$) and the vertices $x_{q_0+\nu}^{-1},x_{q_0+3\nu}^{-1},\ldots , x_{q_0+(n/\nu-1)\nu}^{-1}$ are precisely those negative vertices of $\Gamma_0$ that belong to a different part of $\Gamma_0$ to $x_{q_0}^{-1}$ (and so are neighbours of $x_{q_0}^{-1}$). Hence $\bar{\mathcal{A}}, \bar{\mathcal{B}}$ are sets of the form $\{ \pm \nu,\pm 3\nu, \ldots , \pm (n/(2\nu)-2)\nu , n/2\}$. Moreover, the positive (resp.\,negative) vertices of $\Gamma_0$ induce a complete bipartite graph $K_{n/(2\nu),n/(2\nu)}$, which therefore has $(n/(2\nu))^2$ edges. Since each element of $\bar{\mathcal{A}}'$ contributes $n/\nu$ edges to $\Gamma_0$  and the element $n/2\in \bar{\mathcal{A}} \backslash \bar{\mathcal{A}}'$ contributes $(n/2)/\nu$ edges to $\Gamma_0$, we have $n/\nu (|\bar{\mathcal{A}}|-1)+(n/\nu)/2=$
 $(n/(2\nu))^2$ so $|\bar{\mathcal{A}}'|+1=|\bar{\mathcal{A}}|=((n/\nu)+2)/4$. Similarly $|\bar{\mathcal{B}}'|+1=|\bar{\mathcal{B}}|=((n/\nu)+2)/4$. In particular $n/\nu\equiv 2 \bmod 4$ so parts (a),(b) hold.

Since $x_{q_0}^{-1},x_{q_0+\nu}^{-1}$ belong to different parts of $\Gamma_0$, and there is an edge $x_0-x_{q_0}^{-1}$, the vertices $x_0,x_{q_0+\nu}^{-1}$ belong to the same part of $\Gamma_0$. Hence $\bar{\mathcal{Q}}=\{q_0, q_0+2\nu, \ldots , q_0+(n/\nu-2)\nu \}$. Finally $\mathrm{gcd}(q_0,\nu)$ divides $\mathrm{gcd}(n, a\ (a \in \bar{\mathcal{A}}), b\ (b \in \bar{\mathcal{B}}))=1$, since $P_n(w)$ is irreducible, so $\mathrm{gcd}(q_0,\nu)=1$, and so (c) holds.

Now suppose that the conditions of the statement hold. Then $\Gamma$ has $\nu$ isomorphic components. Consider the component $\Gamma_0$. The set of neighbours of $x_{j\nu}$ ($0\leq j<n/\nu$) is the set
\[\begin{cases}
\{x_{\nu},x_{3\nu}, \ldots , x_{(n/\nu-1)\nu}\} \cup \{x_{q_0}^{-1},x_{q_0+2\nu}^{-1}, \ldots , x_{q_0+(n/\nu-2)\nu}^{-1}\} & \mathrm{if}~j~\mathrm{is~even},\\
\{x_{0},x_{2\nu}, \ldots , x_{(n/\nu-2)\nu}\} \cup \{x_{q_0+\nu}^{-1},x_{q_0+3\nu}^{-1}, \ldots , x_{q_0+(n/\nu-1)\nu}^{-1}\} & \mathrm{if}~j~\mathrm{is~odd}
\end{cases}\]
and so $\Gamma_0$ is bipartite with vertex partition
\begin{alignat*}{1}
&\ \{x_\nu,x_{3\nu},\ldots ,x_{(n/\nu-1)\nu}, x_{q_0}^{-1}, x_{q_0+2\nu}^{-1}, \ldots , x_{q_0+(n/\nu-2)\nu}^{-1} \} \cup\\
&\ \quad \quad \{x_0,x_{2\nu},\ldots ,x_{(n/\nu-2)\nu}, x_{q_0+\nu}^{-1}, x_{q_0+3\nu}^{-1}, \ldots , x_{q_0+(n/\nu-1)\nu}^{-1} \}.
\end{alignat*}
Further, for each $0\leq j<n/\nu$ the set of neighbours $N_\Gamma(x_{q_0+j\nu}^{-1})$ $=N_\Gamma (x_{j\nu})$, so $\Gamma_0$ is isomorphic to $K_{n/\nu,n/\nu}$, as required.
\end{proof}

\bibliographystyle{plain}

\end{document}